\documentclass[11pt,a4paper]{article}

\usepackage{amsmath}
\usepackage{amsthm}
\usepackage{amsfonts}
\usepackage{paralist}
\usepackage{graphicx}
\usepackage{color}

\usepackage{hyperref}
\hypersetup{%
plainpages=false,
colorlinks,urlcolor=blue,linkcolor=blue,citecolor=blue,
pdftitle={Storage for arbitrage and buffering, with applications to
  energy systems},
pdfsubject={},
pdfpagemode={UseOutlines},
pdfpagelayout={OneColumn},
pdfstartview={FitBH}
}

\parskip\smallskipamount
\newtheorem{theorem}{Theorem}
\newtheorem{proposition}[theorem]{Proposition}
\theoremstyle{remark}
\newtheorem{remark}{Remark}

\newcommand{\R}{\mathbb{R}}

\newcommand{\X}{X}
\newcommand{\bA}{\bar{A}}
\newcommand{\Ps}{\mathbf{P}}

\DeclareMathOperator*{\argmin}{arg\,min}

\parskip \smallskipamount

\title{The optimal control of storage for arbitrage
    and buffering, with energy applications}
\author{James Cruise\footnote{Heriot-Watt University.  Research
    supported by EPSRC grant EP/I017054/1}
  \ and Stan Zachary\footnotemark[1]}
\date{\today}

\begin{document}

\maketitle

\begin{abstract}
  We study the optimal control of storage which is used for both
  arbitrage and buffering against unexpected events, with particular
  applications to the control of energy systems in a stochastic and
  typically time-heterogeneous environment.  Our philosophy is that of
  viewing the problem as being formally one of stochastic dynamic
  programming, but of using coupling arguments to provide good
  estimates of the costs of failing to provide necessary levels of
  buffering.  The problem of control then reduces to that of the
  solution, dynamically in time, of a deterministic optimisation
  problem which must be periodically re-solved.  We show that the
  optimal control then proceeds locally in time, in the sense that the
  optimal decision at each time~$t$ depends only on a knowledge of the
  future costs and stochastic evolution of the system for a time
  horizon which typically extends only a little way beyond~$t$.  The
  approach is thus both computationally tractable and suitable for the
  management of systems over indefinitely extended periods of time.
  We develop also the associated strong Lagrangian theory (which may
  be used to assist in the optimal dimensioning of storage), and we
  provide characterisations of optimal control policies.  
  We give examples based on Great Britain electricity price data.
\end{abstract}

\section{Introduction}
\label{sec:introduction}

The control of complex stochastic systems, for example modern power
networks which must cope with many sources of uncertainty in both
generation and demand, requires real-time optimisation of decision
problems which are often computationally intractable---notably so in a
time-heterogeneous environment.  This clearly also poses difficulties
for the design of such systems.  As in the case of the well studied
areas of communication and manufacturing networks, our belief is that
what is required is the careful specification of the stochastic models
governing the behaviour of such systems, coupled with the analytical
derivation of accurate approximation techniques.

In the present paper we use an economic framework to consider the
optimal control of a single storage facility.  The problem is made
interesting because, at least in power networks, storage may be
simultaneously used for many different purposes, with potentially
conflicting objective functions.  However, if storage is to be
economically viable, it must be capable of meeting these competing
objectives.  We concentrate on energy storage in a time-heterogeneous
environment, and consider two of the main uses of such storage
systems: (a) price arbitrage, i.e.\ the buying and selling of energy
over time (whether to earn revenue for the store owner or for the
benefit of the consumer), and (b) the provision of buffering services,
so as to react rapidly to sudden and unexpected changes, for example
the loss of a generator or transmission line, or a sudden surge in
demand.  Our general approach is likely to be applicable to other uses
of storage, and also to the optimal control of other facilities used
for the provision of multiple services.

There is considerable literature on the control of storage for each of
the above two purposes considered on its own.  In the case of the use
of storage for arbitrage, and with linear cost functions for buying
and selling at each instant in time, the problem of optimal control is
the classical \emph{warehouse problem} (see \cite{Cahn, Bell, Drey}
and also \cite{Sec2010} for a more recent example).  Cruise et al
\cite{CFGZ} consider the optimal control of storage---in both a
deterministic and a stochastic setting---in the case where the store
is a price maker (i.e.\ the size of the store is sufficiently large
that its activities influence prices in the market in which it
operates) and is subject to both capacity and rate constraints; they
develop the associated Lagrangian theory, and further show that the
optimal control at any point in time usually depends only on the cost
functions associated with a short future time horizon.  Recent
alternative approaches for studying the value and use of storage for
arbitrage can be found in the
papers~\cite{KHT,PADS,SDJ,VHMS,WAM}---see also the text~\cite{WW}, and
the further references given in~\cite{CFGZ}.  For an assessment of the
potential value of energy storage in the UK electricity system
see~\cite{TMTSBP}.

There have been numerous studies into the use of storage for buffering
against both the increased variability and the increased uncertainty
in electrical power systems, due to higher penetration of renewable
penetration---the former due to the natural variability of such
resources as wind power, and the latter due to the inherent
uncertainty of forecasting.  These studies have considered many
different more detailed objectives; these range from the sizing and
control of storage facilities co-located with the renewable generation
so as to provide a smoother supply and so offset the need for network
reinforcement \cite{CL, DS, KHH}, to studies on storage embedded
within transmission networks so as to increase wind power utilisation
and so reduce overall generation costs \cite{HD,RJM, TO}.  In addition
there have been a number of studies into the more general use of
storage for buffering, for example, so as to provide fast frequency
response to power networks \cite{MMA, OCO, TMTSBP}, or to provide
quality of service as part of a microgrid \cite{BLLBP,HMN}.

In general the problem of using a store for buffering is necessarily
stochastic.  The natural mathematical approach is via stochastic
dynamic programming.  This, however, is liable to be computationally
intractable, especially in the case of long time horizons and the
likely time heterogeneity of the stochastic processes involved.
Therefore much of the literature considers necessarily somewhat
heuristic but nevertheless plausible control policies---again often
adapted to meeting a wide variety of objectives.  For example, for
storage embedded in a distribution network, two control policies are
considered in \cite{BI1}; the first policy aims to feed into a store
only when necessary to keep local voltage levels within a predefined
range and to empty the store again as soon as possible thereafter; the
second policy aims to maintain a constant level of load in the
network.  For larger stores operating within transmission networks,
the buffering policies studied have included that of a fixed target
level policy \cite{BGK}, a dynamic target level policy \cite{GTL}, and
a two stage process with day ahead generation scheduling and a online
procedure to adapt load levels \cite{AA}.

Control policies have been studied via a range analytic and simulation
based methods.  Examples of an analytic approach can be found in
\cite{HPSPB}, where partial differential equations are utilised to
model the behaviour and control of a store, and in \cite{BI1, BI2},
where spectral analysis of wind and load data is used with models
which also incorporate turbine behaviour.  Simulation-based studies
include \cite{BGK, GTL}, which use a bootstrap approach based on real
wind forecast error data, and \cite{AA}, which uses Monte Carlo
simulation of the network state.

In the present paper we study the optimal control of a store which is
used both for arbitrage and for buffering against unpredictable
events.  As previously indicated we use an economic framework, so that
the store sees costs (positive or negative) associated with buying and
selling, and with the provision of buffering services.  The store
seeks to operate in such a way as to minimise over time the sum of
these costs.  We believe such an economic framework to be natural when
the store operates as part of some larger and perhaps very complex
system, provided the price signals under which the store operates are
correctly chosen.  The store may be sufficiently large as to have
market impact, leading to nonlinear cost functions for buying and
selling, may be subject to rate (as well as capacity) constraints,
and, as will typically be the case, may suffer from round-trip
inefficiencies.  We formulate a stochastic model which is realistic in
many circumstances and characterise some of the properties of an
optimal control, relating the results to the existing experimental
literature.  We develop the associated strong Lagrangian theory and,
by making a modest approximation---the validity of which may be tested
in practical applications---show how to construct a computationally
tractable optimal control.  These latter results form a nontrivial
extension of those of the ``arbitrage-only'' case studied
in~\cite{CFGZ}, and require significant new developments of the
necessary optimization theory; as in~\cite{CFGZ}, the optimal control
at any time usually depends on a relatively short time horizon (though
one which is typically somewhat longer than in the earlier case), so
that the algorithm is suitable for the optimal control of the store
over an indefinite period of time.

The optimal control is given by the solution, at the start of the
control period, of a deterministic optimisation problem which can be
regarded as that of minimising the costs associated with the store
buying and selling added to those of notionally ``insuring'' for each
future instant in time against effects of the random fluctuations
resulting from the provision of buffering services.  The cost of such
``insurance'' depends on the absolute level of the store at that time.
Thus this deterministic problem is that of choosing the vector of
successive levels of the store so as to minimise a cost function
$\sum_t[C_t(x_t)+A_t(s_t)]$, subject to rate and capacity constraints,
where $C_t(x_t)$ is the cost of incrementing the level of the store
(positively or negatively) at time~$t$ by $x_t$, and the ``penalty''
function~$A_t$ is such that $A_t(s_t)$ is the expected cost of any
failure to provide the required buffering services at the time~$t$
when the level of the store is then $s_t$.  We define this
optimisation problem~$\Ps$ more carefully in
Sections~\ref{sec:problem} and \ref{sec:simpl-optim-probl}.  In the
stochastic environment in which the store operates, the solution of
this deterministic problem determines the future control of the store
until such time as its buffering services are actually required,
following which the level of the store is perturbed and the
optimisation problem must be re-solved starting at the new level.  The
continuation of this process provides what is in principle the exactly
optimal stochastic control of the store on a potentially indefinite
time scale.

In Section~\ref{sec:problem} we formulate the relevant stochastic
model and discuss its applicability.  This enables us, in
Section~\ref{sec:char-optim-solut} to provide some characteristic
properties of optimal solutions, which we relate to empirical work in
the existing literature.  In Sections~\ref{sec:determ-funct-a_t-2} and
\ref{sec:simpl-optim-probl} we develop the approach to an optimal
control outlined above.  Section~\ref{sec:solution} considers the
deterministic optimisation problem associated with the stochastic
control problem and derives the associated strong Lagrangian theory,
while in Section~\ref{sec:algorithm} we develop an efficient
algorithm.  Section~\ref{sec:example} gives examples.

\section{Problem formulation}
\label{sec:problem}

Consider the management of a store over a finite time interval $[0,T]$
where the time horizon $T$ is integer, and where $[0,T]$ is divided
into a succession of periods $t=1,\dots,T$ of integer length.  At the
start of each time period~$t$ the store makes a decision as to how
much to buy or sell during that time period; however, the level of the
store at the end of that time period may be different from that
planned if, during the course of the period, the store is called upon
to provide buffering services to deal with some unexpected problem or
\emph{shock}.  Such a shock might be the need to supply additional
energy during the time period~$t$ due to some unexpected failure---for
example that of a generator---or might simply be the difference
between forecast and actual renewable generation or demand.  We
suppose that the capacity of the store during the time period~$t$ is
$E_t$ units of energy.  (Usually $E_t$ will be constant over time, but
need not be, and there are some advantages---see in particular
Section~\ref{sec:determ-funct-a_t-2}---in allowing the time
dependence.)  Similarly we suppose that the total energy which may be
input or output during the time period~$t$ is subject to rate (i.e.\
power) constraints $P_{It}$ and $P_{Ot}$ respectively.  This
slotted-time model corresponds, for example, to real world energy
markets where energy is typically traded at half-hourly or hourly
intervals, with the actual delivery of that energy occurring in the
intervening continuous time periods.  Detailed descriptions of
the operation of the UK market can be found in \cite{NAO,ELEXON}.

For each~$t$ let $\X_t=\{x:-P_{Ot}\le x\le P_{It}\}$.  Both buying and
selling prices associated with any time period~$t$ may be represented
by a convex function~$C_t$ defined on $\X_t$ which is such that, for
positive $x$, $C_t(x)$ is the price of buying $x$ units of energy for
delivery during the time period~$t$, while, for negative $x$, $C_t(x)$
is the negative of the price for selling $-x$ units of energy during
that time period.  Thus, in either case, $C_t(x)$ is the cost of a
planned change of $x$ to the level of the store during the time
period~$t$, in the absence of any buffering services being required
during the course of that time period.  The convexity assumption
corresponds, for each time~$t$, to an increasing cost to the store of
buying each additional unit, a decreasing revenue obtained for selling
each additional unit, and every unit buying price being at least as
great as every unit selling price.  When, as is usually the case, the
store is not perfectly \emph{efficient} in the sense that only a
fraction $\eta\le1$ of the energy input in available for output, then
this may be captured in the cost function by reducing selling prices
by the factor~$\eta$; under the additional assumption that the cost
functions~$C_t$ are increasing it is easily verified that this
adjustment preserves the above convexity of the functions~$C_t$.  We
thus assume that the cost functions are so adjusted so as to capture
any such round-trip inefficiency.

\begin{remark}
  \label{rmk:1}
  A further form of possible inefficiency of a store is
  \emph{leakage}, whereby a fraction of the contents of the store is
  lost in each unit of time.  We do not explicitly model this here.
  However, only routine modifications are required to do so, and are
  entirely analogous to those described in~\cite{CFGZ}.
\end{remark}
\begin{remark}
  \label{rmk:2}
  Note also that, in the above model, it is possible to absorb the
  rate constraints into the cost functions---by setting the costs
  associated with $x\notin\X_t$ to be prohibitively high---and to
  preserve the convexity of these functions.  However, in general we
  prefer to avoid this approach here.
\end{remark}

Suppose that at the end of the time period $t-1$, or equivalently at
the start of the time period~$t$, the level of the store is $s_{t-1}$
(where we take $s_0$ to be the initial level of the store).  We assume
that one may then choose a \emph{planned} adjustment (contract to buy
or sell) $x_t\in\X_t$---and such that additionally
$s_{t-1}+x_t\in[0,E_t]$---to the level of the store during the time
period~$t$, the cost of this adjustment being $C_t(x_t)$.  Subsequent
to this, during the course of the time period~$t$, the the store may
subject to some \emph{shock} or random disturbance, corresponding
perhaps to the need to provide unexpected buffering services, which
may both disturb the final level of the store at the end of that time
period---and perhaps also at the end of subsequent time periods---and
have further associated costs, the latter being typically those of the
store not being able to provide the required services.

For each $t$, and for each possible level~$s_{t-1}$ of the store at
the end of the time period~$t-1$, define $V_{t-1}(s_{t-1})$ to be the
expected future cost of subsequently managing the store under an
optimal strategy (i.e.\ one under which this expected cost is
minimised), under the assumption that either no shocks have occurred
by the end of the time period~$t-1$ or that, given the level
$s_{t-1}$, such past shocks as have occurred by that time do not
influence the optimal future management of the store or its associated
costs.  Under these conditions, and for a planned adjustment~$x_t$ to
the level of the store during the time period~$t$ (at an immediate
cost~$C_t(x_t)$ as indicated above), in the absence of any shock
during the time period~$t$, the expected cost of optimally managing
the store thereafter is then $V_t(s_{t-1}+x_t)$.  We assume that the
expected \emph{additional} cost to the store, both immediate and
future, of dealing optimally with any shock which may occur during the
time period~$t$ is a function~$A_t(s_{t-1}+x_t)$ of the planned
level~$s_{t-1}+x_t$ of the store for the end of the time period~$t$.
We then have that
\begin{equation}
  \label{eq:1}
    V_{t-1}(s_{t-1}) = \min_{\substack{x_t\in \X_t\\s_{t-1}+x_t\in\cap[0,E_t]}}
  \left[
    C_t(x_t) + A_t(s_{t-1}+x_t) + V_t(s_{t-1}+x_t)
  \right],
\end{equation}
and that the optimal planned increment to the level of the store for
the time period~$t$ (given that an optimal policy is to be followed
thereafter) is given by $\hat{x}_t(s_{t-1})$ where this is defined to
be the value of $x_t\in\X_t$ which achieves the minimisation in the
recursion~\eqref{eq:1}.

We also define the terminal condition
\begin{equation}
  \label{eq:2}
  V_T(s_T) = 0
\end{equation}
for all possible levels $s_T$ of the store at the end of the time
period~$T$.

Note that $A_t(s_{t-1}+x_t)$ (which may be alternatively be
interpreted as the ``insurance'' cost associated with the planed level
of the store for the time period~$t$ as described in the Introduction)
may be understood via a coupling argument, in which the possibly
disturbed and subsequently optimally controlled process of store
levels---following any shock in the time period~$t$---is coupled to
the process which is undisturbed in that time period and subsequently
optimally controlled; $A_t(s_{t-1}+x_t)$ is then the expected
difference in the costs of operating the two processes until such time
(if ever) as they subsequently merge.  As we discuss further below,
this interpretation proves useful in finding workable approximations
to the functions~$A_t$.

\begin{remark}
  We make the assumption above that each function~$A_t$, representing
  the extra cost of dealing with a shock occurring during the time
  period~$t$, may be represented as a function of the planned level
  $s_{t-1}+x_t$ of the store for the end of that time period and,
  given this, does not further depend on the level~$s_{t-1}$ of the
  store at the beginning of that time period.  The accuracy of this
  assumption will vary according to the precise characteristics of the
  store, the way in which it interacts with its external environment
  in the event of shocks, and the various cost functions which form
  part of the model.  The assumption is likely to be at its most
  accurate when rate constraints do not play a major role in the
  management of the store, as the store may adjust to its target
  levels quickly.  Elsewhere, when the level of the store does not
  change too much during a single time period, the assumption may
  still be regarded as a reasonable approximation.  Its
  relaxation---for example by allowing $A_t$ to be a more general
  function of $s_{t-1}$ and $x_t$---simply complicates without
  essentially changing the analysis below.
\end{remark}

Our aim is to determine the optimal control of the store over the time
interval $[0,T]$.  Such a control will necessarily be stochastic.  In
principle some form of stochastic dynamic programming approach is
required.  However, particularly within a time heterogeneous
environment (in which there may be no form of regularity in either the
functions~$C_t$ or in the shock processes), such an approach would be
unlikely to be efficient, and might well prove computationally
intractable, on account of (a) the need, in such an approach, to
completely determine each of the functions~$V_t$ defined above, and
(b) the need to solve the problem over the entire time interval
$[t,T]$ in order to determine the optimal control at any time~$t$.

Our method of proceeding is therefore as follows.  We assume that the
functions $A_t$ are known, at least to within reasonable
approximations.  (We argue in Section~\ref{sec:determ-funct-a_t-2}
that in many cases the functions~$A_t$ may be determined either
exactly or to within a very good approximation; this follows from the
coupling characterisation of these functions introduced above.)  Given
the initial level~$s_0$ of the store we may then use the argument
leading to the recursion~\eqref{eq:1} and \eqref{eq:2} to determine
very efficiently a control which remains optimal up to the end of the
first time period in which a shock actually occurs.  Following such a
shock (and, if necessary, once its knock-on effects have cleared from
the system---again see the discussion of
Section~\ref{sec:determ-funct-a_t-2}), the current state (level) of
the store is reexamined and the optimal future control strategy
recalculated.  Iteration of this process leads to an efficient
(stochastic) dynamic control for the entire time interval $[0,T]$.  We
also show below that typically the optimal decision at (the start of)
any time~$t$ depends only on the functions $C_{t'}$ and $A_{t'}$ for
values of time~$t'$ extending only a little beyond the time~$t$.  The
approach outlined above is therefore generally also suitable for the
ongoing optimal management of the store over an indefinite period of
time.




\section{Characterisation of optimal solutions}
\label{sec:char-optim-solut}

In this section we establish some properties of the functions
$\hat{x}_t(\cdot)$ defined in the previous section and determining the
optimal control of the store.

One case which will be of particular interest is that where the store
is a price taker (i.e.\ the store is not so large as to impact
itself on market prices), so that, for each~$t$, the cost
function~$C_t$ is given by
\begin{equation}
  \label{eq:3}
  C_t(x) =
  \begin{cases}
    c^{(b)}_t x, & \quad\text{if $x\ge0$}\\
    c^{(s)}_t x, & \quad\text{if $x<0$}.
  \end{cases}
\end{equation}
and where the unit ``buying'' price~$c^{(b)}_t$ and the unit
``selling'' price~$c^{(s)}_t$ are such that $c^{(s)}_t\le c^{(b)}_t$.
(That, at any time~$t$, the reward obtained in the market resulting
from decreasing the level of the store by a single unit may be less
than the cost of increasing the level of the store by a single unit
may primarily reflect the fact that the store may be less than
perfectly efficient---see the discussion of
Section~\ref{sec:problem}.)

Proposition~\ref{prop:simple} below is a very simple result which
shows that in the case where buying and selling prices are equal
(typically corresponding to a perfectly efficient store), and provided
rate constraints are nonbinding, the optimal policy is a ``target''
one.  By this we mean that for each time period~$t$ there exists a
target level~$\hat{s}_t$: given that the level of the store at the end
of the immediately preceding time period is $s_{t-1}$ and that shocks
prior to that time have no further ongoing effects on the management
of the store, the optimal planned level~$s_{t-1}+x_t$ of the store to
be achieved during the time period~$t$ is set to some
value~$\hat{s}_t$, independently of $s_{t-1}$.

\begin{proposition}\label{prop:simple}
  Suppose that, for each $t$, we have $c^{(b)}_t=c^{(s)}_t=c_t$ say;
  define
  \begin{equation}
    \label{eq:4}
    \hat{s}_t = \argmin_{s\in[0,E_t]}[c_t s + A_t(s) + V_t(s)].
  \end{equation}
  Then, for each~$t$ and for each $s_{t-1}$, we have
  $\hat{x}_t(s_{t-1})=\hat{s}_t-s_{t-1}$ provided only that this
  quantity belongs to the set $\X_t$.
\end{proposition}

\begin{proof}
  The recursion~\eqref{eq:1} here becomes, for each $t$,
  \begin{equation}
    \label{eq:5}
    V_{t-1}(s_{t-1}) = \min_{\substack{x_t\in \X_t\\s_{t-1}+x_t\in\cap[0,E_t]}}
    \left[
      c_tx_t + A_t(s_{t-1}+x_t) + V_t(s_{t-1}+x_t)
    \right],
  \end{equation}
  and the above minimisation is achieved by $x_t$ such that
  $s_{t-1}+x_t=\hat{s}_t$, provided only that $x_t\in\X_t$.
\end{proof}

In order to deal with the possibility of rate constraint violation,
with the more general price-taker case where $c^{(s)}_t<c^{(b)}_t$,
and with the quite general case where the cost functions~$C_t$ are
merely required to be convex, we require the additional assumption of
convexity of the functions~$A_t$.  This latter condition, while not
automatic, is reasonably natural in many applications---see the
examples of Section~\ref{sec:example}.

\begin{proposition}\label{prop:aconvex}
  Suppose that, in addition to convexity of the functions~$C_t$, each
  of the functions~$A_t$ is convex.  Then, for each~$t$:
  \begin{compactenum}[(i)]
  \item the function~$V_{t-1}$ is convex;
  \item $\hat{x}_t(s_{t-1})$ is a decreasing function of $s_{t-1}$;
  \item $s_{t-1}+\hat{x}_t(s_{t-1})$ is an increasing function of
    $s_{t-1}$.
  \end{compactenum}
\end{proposition}

\begin{proof}
  It is helpful to define, for each $t=1,\dots,T$, the
  function~$U_{t-1}$ of each possible level~$s_{t-1}$ of the store at
  the end of the time period~$t-1$, and each possible planned
  increment~$x_t$ to the level of the store for the time period~$t$ by
  \begin{equation}
    \label{eq:6}
    U_{t-1}(s_{t-1},\,x_t) = C_t(x_t) + A_t(s_{t-1}+x_t) + V_t(s_{t-1}+x_t).
  \end{equation}
  The recursion~\eqref{eq:1} now becomes
  \begin{equation}
    \label{eq:7}
    V_{t-1}(s_{t-1})
    = \min_{\substack{x_t\in \X_t\\s_{t-1}+x_t\in\cap[0,E_t]}} U_{t-1}(s_{t-1},\,x_t).
  \end{equation}
  To show (i) we use backwards induction in time.  The function $V_T$
  is convex.  Suppose that, for any given $t\le T$, the function~$V_t$
  is convex; we show that the function~$V_{t-1}$ is convex.  For each
  of given values $s^{(i)}_{t-1}$, $i=1,\dots,n$ of $s_{t-1}$, let
  $x^{(i)}_t$ be the value of $x_t$ which achieves the minimisation in
  \eqref{eq:7}, and for any convex combination
  $\bar s_{t-1}=\sum_{i=1}^n\kappa_is^{(i)}_{t-1}$, where each
  $\kappa_i\ge0$ and where $\sum_{i=1}^n\kappa_i=1$, define also
  $\bar{x}_t=\sum_{i=1}^n\kappa_ix^{(i)}_t$.  Note that
  $\bar{x}_t\in\X_t$ and that $\bar{s}_{t-1}+\bar{x}_t\in[0,E_t]$.  Then,
  from~\eqref{eq:7},
  \begin{align*}
    V_{t-1}(\bar s_{t-1})
    & \le U_{t-1}(\bar s_{t-1},\,\bar x_t)\\
    &  \le \sum_{i=1}^n\kappa_i U_{t-1}(s^{(i)}_{t-1},\,x^{(i)}_t)\\
    & = \sum_{i=1}^n\kappa_iV_{t-1}(s^{(i)}_{t-1}),
  \end{align*}
  where the second line in the above display follows from the
  definition~\eqref{eq:6} of the function~$U_{t-1}$ and the convexity
  of the functions~$C_t$, $A_t$ and $V_t$ (the latter by the inductive
  hypothesis).  Thus $V_{t-1}$ is convex as required.

  To show (ii) and (iii), given values
  $s^{(1)}_{t-1}\le s^{(2)}_{t-1}$ of $s_{t-1}$, again let
  $x^{(1)}_t$, $x^{(2)}_t$ be the respective values of $x_t$ which
  achieves the minimisation in \eqref{eq:7}.  Since for the function
  $U_{t-1}(s^{(1)}_{t-1},\,\cdot)$ is minimised in $\X_t\cap E_t$ at
  $x^{(1)}_t$, it follows straightforwardly, from the
  definition~\eqref{eq:6} of the function~$U_{t-1}$ and the convexity
  of the function~$C_t$ and that of the function $A_t+V_t$, that,
  since $s^{(1)}_{t-1}\le s^{(2)}_{t-1}$, the minimisation of the
  function $U_{t-1}(s^{(2)}_{t-1},\,\cdot)$ is achieved (or, in the
  case of nonuniqueness, may be achieved) at $x^{(2)}_t\le x^{(1)}_t$.
  Thus the result~(ii) follows.  Similarly, it is again
  straightforward from the convexity of the function $C_t$ and that of
  the function $A_t+V_t$ and since $s^{(1)}_{t-1}\le s^{(2)}_{t-1}$,
  that $x^{(2)}_t$ is (or, in the case of nonuniqueness, may be taken
  to be) such that
  $s^{(2)}_{t-1}+x^{(2)}_t\ge s^{(1)}_{t-1}+x^{(1)}_t$.  The
  result~(iii) thus similarly follows.
\end{proof}

\begin{remark}
  \label{rmk:3}
  Note that the rate constraints $x_t\in\X_t$ for all $t$ cause no
  difficulties for the above proof---a result which may alternatively
  be seen by absorbing these constraints into the cost functions $C_t$
  as described in Remark~\ref{rmk:2}.
\end{remark}

We now return to the price-taker case, in which the cost functions
are as defined by~\eqref{eq:3}, and which corresponds to a store which
is not sufficiently large as to have market impact.  Here we may prove
a strengthened version of Proposition~\ref{prop:aconvex}.  For each
$t$, given that the function~$A_t$ is convex, define
\begin{equation}
  \label{eq:8}
  s^{(b)}_t = \argmin_{s\in[0,E_t]}[c^{(b)}_t s + A_t(s) + V_t(s)]
\end{equation}
and similarly define
\begin{equation}
\label{eq:9}
  s^{(s)}_t = \argmin_{s\in[0,E_t]}[c^{(s)}_t s + A_t(s) + V_t(s)].
\end{equation}
Note that the above convexity assumption and the condition that, for
each~$t$, we have $c^{(s)}_t\le c^{(b)}_t$ imply that $s^{(b)}_t\le
s^{(s)}_t$.
We now have the following result.
\begin{proposition} 
  \label{prop:interval}
  Suppose that the cost functions~$C_t$ are as given by \eqref{eq:3}
  and that the functions~$A_t$ are convex.  Then the optimal policy is
  given by: for each~$t$ and given $s_{t-1}$, take
  \begin{equation}
    \label{eq:10}
    x_t =
    \begin{cases}
      \min(s^{(b)}_t-s_{t-1},\,P_{It}) & \quad\text{if $s_{t-1}<s^{(b)}_t$,}\\
      0  & \quad\text{if $s^{(b)}_t\le s_{t-1}\le s^{(s)}_t$,}\\
      \max(s^{(s)}_t-s_{t-1},\,-P_{Ot}) & \quad\text{if $s_{t-1}>s^{(s)}_t$.}
    \end{cases}
  \end{equation}
\end{proposition}

\begin{proof}
  For each $t$, it follows from the convexity of the functions~$C_t$,
  $A_t$ and $V_t$ (the latter by the first part of
  Proposition~\ref{prop:aconvex}) that, for $s_{t-1}<s^{(b)}_t$ the
  function $C_t(x_t)+A_t(s_{t-1}+x_t)+V_t(s_{t-1}+x_t)$ is minimised
  by $x_t=s^{(b)}_t-s_{t-1}$, for $s^{(b)}_t\le s_{t-1}\le s^{(s)}_t$
  it is minimised by $x_t=0$, while for $s_{t-1}>s^{(s)}_t$, it is
  minimised by $x_t=s^{(s)}_t-s_{t-1}$.  The required result now
  follows from the recursion~\eqref{eq:1}.
\end{proof}

Thus in general in the price-taker case there exists, for each time
period~$t$, a ``target interval'' $[s^{(b)}_t,s^{(s)}_t]$ such that,
if the level of the store at the end of the previous time period is
$s_{t-1}$ (and again given that the shocks prior to this time have no
ongoing effects on the optimal management of the store), the optimal
policy is to chose $x_t$ so that $s_{t-1}+x_t$ is the nearest point
(in absolute distance) to $s_{t-1}$ lying within, or as close as
possible to, the above interval.  In the case where
$c^{(b)}_t=c^{(s)}_t=c_t$, the above interval shrinks to the single
point~$\hat{s}_t$ defined by~\eqref{eq:4}.

These results shed some light on earlier, more applied, papers of
Bejan et al~\cite{BGK} and Gast et al~ \cite{GTL}, in which the
uncertainties in the operation of a energy store result from errors in
wind power forecasts.  The model considered in those papers is close
to that of the present paper, as we now describe.  The costs of
operating the store result (a) from round-trip inefficiency, which in
the formulation of the present paper would be captured by the cost
functions $C_t$ as defined by~\eqref{eq:3} with $C_t$ the same for
all~$t$, and (b) from buffering events, i.e.\ from failures to meet
demand through insufficient energy available to be supplied from the
store when it is needed, and from energy losses through store
overflows.  In the formulation of the present paper these costs would
be captured by the functions~$A_t$.  In contrast to the present paper
decisions affecting the level of the store (the amount of conventional
generation to schedule for a particular time) are made $n$ time
steps---rather than a single time step---in advance.  The shocks to
the system result from the differences between the available wind
power as forecasted $n$ steps ahead of real time (when conventional
generation is scheduled) and the wind power actually obtained.
Although the model of the above papers is therefore not exactly the
same as that of the present paper, the underlying arguments leading to
Propositions~\ref{prop:simple}--\ref{prop:interval} continue to apply,
at least to a good approximation.  In particular sample path arguments
suggest that the reduction of round-trip efficiency slows the rate at
which the store-level trajectories---started from different initial
levels but with the same stochastic description of future shock
processes---converge over subsequent time.


Bejan et al~\cite{BGK} consider only the case where the round-trip
efficiency is $1$.  They study the efficiency of policies---analogous
to those suggested by Proposition~\ref{prop:simple}---whereby, for
each~$t$, the generation scheduled for time~$t$ at the earlier time
$t-n$ is such as would, given perfect forecasting, achieve a given
target level $\hat{s}_t$ of the store at time~$t$; this target level
is independent of the level $s_{t-n}$ of the store at the time $t-n$
and of earlier scheduling decisions.  However, Bejan et al~\cite{BGK}
further take $\hat{s}_t$ to be independent of $t$, something which may
not be optimal given the likely nonstationarity of the process of
forecast errors.


Gast et al~\cite{GTL} subsequently study the same time series of
available wind power, but allowed for round-trip efficiencies which
are less than~$1$.  They find (as might be expected here) that simple
``target'' policies such as that described above do not work well
under these circumstances, and compare the behaviour of a variety of
time-homogeneous policies.

\section{Determination of the functions~$A_t$}
\label{sec:determ-funct-a_t-2}

We described in Section~\ref{sec:problem} how, given a knowledge of
the functions~$A_t$, the optimal control of the store could be
determined.  In
Sections~\ref{sec:simpl-optim-probl}--\ref{sec:algorithm} we develop
such an approach, which is based on strong Lagrangian theory and which
is very much more efficient, in senses explained there, than the
application of standard dynamic programming or nonlinear optimisation
techniques.  In this section we consider conditions under which the
functions~$A_t$ may be thus known, either exactly or to good
approximations.

Suppose that, as in Section~\ref{sec:problem}, at the end of the time
period~$t-1$ the level of the store is $s_{t-1}$ and that, given
$s_{t-1}$, any shocks prior to that time have no further effect on the
optimal management of the store.  Suppose further that an increase of
$x_t$ (positive or negative) is planned for the time period~$t$ (at a
cost of $C_t(x_t)$).  Recall that $A_t(s_{t-1}+x_t)$ is then defined
to be the expected additional cost to the store of dealing optimally
with any shock which may occur during the time period~$t$, and may be
conveniently characterised in terms of the coupling defined in that
Section~\ref{sec:problem}.  Now define also $\bA_t(s_{t-1}+x_t)$ to be
the expected additional cost to the store of dealing with any shock
which may occur during the time period~$t$ and \emph{immediately}
returning the level of the store to its planned level $s_{t-1}+x_t$ at
the end of the time period~$t$.  As in the case of the function~$A_t$,
we assume that each function $\bA_t$ depends on $s_{t-1}$ and $x_t$
through their sum $s_{t-1}+x_t$---the extent to which this
approximation is reasonable being as discussed for the
functions~$A_t$.  Given the costs of dealing with any shocks, and the
known costs of making any immediate subsequent adjustments to the
level of the store, the functions $\bA_t$ are readily determinable,
and in particular do not depend on how the store is controlled outside
the time period~$t$.

Note that, in the case of linear cost functions (i.e.\ $C_t(x)=c_tx$
for all~$t$) and when shocks do not have effects which persist beyond
the end of the time period in which they occur, the argument of
Proposition~\ref{prop:simple} implies immediately that $A_t=\bA_t$ for
all~$t$: the linearity of $C_t$ implies that, at the end of the time
period~$t-1$ and when the level of the store is then $s_{t-1}$, if
$s_{t-1}+x_t$ is the optimal planned level of the store for the end of
the time period~$t$, then it remains the optimal level of the store
for the end of that time period following any shock which occurs
during it.

More generally the functions $\bA_t$ provide reasonable approximations
to the functions~$A_t$ to the extent to which it is reasonable,
following any shock with which the store is required to deal, to
return immediately the level of the store to that which would have
obtained in the absence of the shock.  In particular, when shocks are
relatively rare but are potentially expensive (as might be the case
when the store is required to pay the costs of failing to have
sufficient energy to deal with an emergency), then the major
contribution to both the functions~$A_t$ and $\bA_t$ will be this
cost, regardless of precisely how the level of the store is adjusted
in the immediate aftermath of the shock. 

If necessary, better approximations to the functions $A_t$ may be
obtained by allowing longer periods of time in which to optimally
couple the trajectory of the store level, following a shock, to that
which would have obtained in its absence.  In applications one would
wish to experiment a little here.

In applications there is also a need, when the costs of a shock arise
from a failure to have insufficient energy in the store to deal with
it, to identify what these costs are.  There are various possible
candidates.  Two simple such---natural in the context of risk metrics
for power systems, where they correspond respectively to \emph{loss of
  load} and \emph{energy unserved} (see, for example,
\cite{Billinton})---are:
\begin{compactenum}[(i)]
\item for each $t>0$, the cost of a shock occurring during the time
  period~$t$ is simply some constant $a_t>0$ if there is insufficient
  energy within the store to meet it, and is otherwise~$0$.
\item for each $t>0$, the cost of a shock occurring during the time
  period~$t$ is proportional to the shortfall in the energy necessary
  to meet that shock.
\end{compactenum}
Given the planned level $s_{t-1}+x_t$ of the store to be achieved
during any time period~$t$, the total additional cost of dealing with
any shock occurring during that time period (as defined for example in
terms of the coupling introduced in Section~\ref{sec:problem}) is a
random variable which is a function of the size of the shock.  The
distribution of this random variable, and so also its expectation
$A_t(s_{t-1}+x_t)$ may need to be determined by observation.

Note finally that the effects of shocks may persist over several time
periods (as, for example, when the store is required to provide
ongoing support for the sudden loss of major piece of equipment such
as a generator), so that each of the functions~$A_t$---which will in
general be decreasing---need not be flat for values of its argument in
excess of the output rate constraint~$P_{Ot}$.  In particular a
reasonable way of dealing with a shock whose effects do persist over
several time periods may simply be to reserve notionally sufficient
energy in the store to deal with it; then, following such a shock, the
level of the store will temporally become the excess over that reserve
and the capacity of the store will correspondingly be temporally
reduced.  This causes no problems for the present theory, and is a
reason for allowing a possible time dependence (which may be dynamic)
for the capacity of the store.

We consider some plausible functional forms of the functions~$A_t$ in
Section~\ref{sec:example}.

\section{The optimal control problem}
\label{sec:simpl-optim-probl}

We now assume that the functions $A_t$ defined in
Section~\ref{sec:problem} are known, at least to a sufficiently good
approximation---see the discussion of the previous section.

Define (the random variable) $\hat{s}=(\hat{s}_0,\dots,\hat{s}_T)$
(with $\hat{s}_0=s^*_0$) to be the levels of the store at the end of
the successive time periods $t=0,\dots,T$ under the (stochastic)
optimal control as defined in Section~\ref{sec:problem}.
Recall also from Section~\ref{sec:problem} that, for each~$t$ and each
level $s_{t-1}$ of the store at the end of the time period~$t-1$, the
quantity~$\hat{x}_t(s_{t-1})$ is the value of $x_t\in\X_t$ which
achieves the minimisation in the recursion~\eqref{eq:1}.

For any vector $s=(s_0,\dots,s_T)$ and for each $t=1,\dots,T$, define
\begin{equation}
  \label{eq:13}
  x_t(s) = s_t - s_{t-1}.
\end{equation}
Define also the following (deterministic) optimisation problem:
\begin{compactitem}
\item[$\Ps$:]
  choose $s=(s_0,\dots,s_T)$ with $s_0=s^*_0$ so as to minimise
  \begin{equation}
    \label{eq:14}
    \sum_{t=1}^T [C_t(x_t(s)) + A_t(s_t)]
  \end{equation}
  subject to the capacity constraints
  \begin{gather}
    \label{eq:15}
    0 \le s_t\le E_t,
    \quad 1 \le t \le T,
  \end{gather}
  and the rate constraints
  \begin{equation}
    \label{eq:16}
    x_t(s) \in \X_t,
    \qquad 1 \le t \le T.
  \end{equation}
\end{compactitem}
Let $s^*=(s^*_0,\dots,s^*_T)$ denote the solution to the above
problem~$\Ps$.  It follows from direct iteration of the
recursion~\eqref{eq:1}, using also~\eqref{eq:2}, that $x_1(s^*)$
achieves the minimisation in~\eqref{eq:1} for $t=1$ and when
$s_0=s^*_0$, i.e.\ that
$\hat{x}_1(s^*_0)=\hat{x}_1(\hat{s}_0)=x_1(s^*)$.  Thus,
from~\eqref{eq:13}, provided no shock occurs during the time
period~$1$ so that $\hat{s}_1=\hat{s}_0+\hat{x}_1(\hat{s}_0)$, we have
also that $\hat{s}_1=s^*_1$.  More generally, let the random
variable~$T'$ index the first time period during which a shock does
occur.  Then repeated application of the above argument gives
immediately the following result.

\begin{proposition}
  For all $t<T'$, we have $\hat{s}_t=s^*_t$.
\end{proposition}



The solution to the problem $\Ps$ therefore defines the optimal
control of the store up to the end of the time period~$T'$ defined
above.  At that time, and the end of each subsequent time period
during which there occurs a shock, it is of course necessary to
reformulate the problem~$\Ps$, starting at the end of the time
period~$T'$ (or as soon any shock occurring during that time period
has been fully dealt with), instead of at time~$0$, and replacing the
initial level $s^*_0=\hat{s}_0$ by the perturbed level $\hat{s}_{T'}$
of the store at that time.  Thus the stochastic optimal control
problem may be solved dynamically by the solution of the
problem~$\mathbf{P}$ at time~$0$, and the further solution of (a
reformulated version) of this problem at the end of each subsequent
time period in which a shock occurs.  The solution of the problem,
which we now consider, is very much simpler than that of the
corresponding stochastic dynamic programming approach.


\section{Lagrangian theory and characterisation of solution}
\label{sec:solution}

We showed in the previous section that, to the extent that the
functions~$A_t$ are known, an optimal control for the store may be
developed via the solution of the optimisation problem~$\Ps$ defined
there.  In Section~\ref{sec:determ-funct-a_t-2} we discussed how to
make what are in many cases good and readily determinable
approximations for the functions~$A_t$.

We again assume convexity of the functions~$A_t$ (see
Section~\ref{sec:char-optim-solut}), in addition to that of the
functions~$C_t$.  We develop the strong Lagrangian
theory~\cite{Boyd,Whi} associated with the problem~$\Ps$.  This leads
to both an efficient algorithm for its solution, and to the
identification of the Lagrange multipliers necessary for the proper
dimensioning of the store.  In particular Theorem~\ref{thm:exist}
establishes the existence of a pair of vectors $(s^*,\,\lambda^*)$
such that $s^*$ solves the problem~$\Ps$ and $\lambda^*$ is a function
of the associated Lagrange multipliers corresponding to the capacity
constraints (see below); the theorem further gives conditions
necessarily satisfied by the pair $(s^*,\,\lambda^*)$.

We now introduce the more general problem~$\mathbf{P}(a,\,b)$ in which
$s_0$ is kept fixed at the value $s^*_0$ of interest above, but in
which $s_1,\dots,s_T$ are allowed to vary between quite general upper
and lower bounds:
\begin{compactitem}
\item[$\mathbf{P}(a,\,b)$:]
  choose $s=(s_0,\dots,s_T)$, with $s_0=s^*_0$ so as to minimise
  \begin{equation}
    \label{eq:17}
    \sum_{t=1}^T [C_t(x_t(s)) + A_t(s_t)]
  \end{equation}
  subject to the capacity constraints
  \begin{equation}
    \label{eq:18}
    a_t \le s_t\le b_t,
    \quad 1 \le t \le T,
  \end{equation}
  and the rate constraints
  \begin{equation}
    \label{eq:19}
    x_t(s) \in \X_t,
    \qquad 1 \le t \le T.
  \end{equation}
\end{compactitem}
Here $a=(a_1,\dots,a_T)$ and $b=(b_1,\dots,b_T)$ are such that
$a_t\le b_t$ for all $t$.  Let also $a^*$ and $b^*$ be the values of
$a$ and $b$ corresponding to our particular problem~$\Ps$ of interest,
i.e.\
\begin{equation}
  \label{eq:20}
  a^*_t = 0, \quad b^*_t = E_t, \quad 1\le t\le T.
\end{equation}


Note that the convexity of the functions $C_t$ and $A_t$ guarantees
their continuity, and, since for each $a$, $b$ as above the space of
allowed values of $s$ is compact, a solution $s^*(a,\,b)$ to the
problem~$\mathbf{P}(a,\,b)$ always exists.  Let $V(a,\,b)$ be the
corresponding minimised value of the objective function, i.e.\
\begin{displaymath}
  V(a,\,b)=\sum_{t=1}^T[C_t(x_t(s^*(a,\,b)))+A_t(s^*(a,\,b))].
\end{displaymath}
Observe also that the function $V(a,\,b)$ is itself convex in $a$ and
$b$.  To see this, consider any convex combination
$(\bar a,\bar b)=(\kappa a_1+(1-\kappa)a_2,\kappa b_1+(1-\kappa)b_2)$
of any two values $(a_1,b_1)$, $(a_2,b_2)$ of the pair $(a,b)$, where
$0\le\kappa\le1$; since the constraints~\eqref{eq:18} and \eqref{eq:19}
are linear, it follows that the vector
$\bar s=\kappa s^*(a_1,b_1)+(1-\kappa)s^*(a_2,b_2)$ is feasible for
the problem~$\mathbf{P}(\bar a,\,\bar b)$; hence, from the convexity
of the functions~$C_t$ and $A_t$,
\begin{align*}
  V(\bar a,\bar b)
  & \le \sum_{t=1}^T[C_t(x_t(\bar s)) + A_t(\bar s_t)]\\
  & = \sum_{t=1}^T[C_t(\kappa x_t(s^*(a_1,b_1)) + (1-\kappa)x_t(s^*(a_2,b_2)))
  + A_t(\kappa s^*_t(a_1,b_1) + (1-\kappa)s^*_t(a_2,b_2))]\\
  & \le \kappa\sum_{t=1}^T[C_t(x_t(s^*(a_1,b_1)))+A_t(s^*_t(a_1,b_1))]
  + (1-\kappa)\sum_{t=1}^T[C_t(x_t(s^*(a_2,b_2)))+A_t(s^*_t(a_2,b_2))]\\
  & = \kappa V(a_1,b_1) + (1-\kappa)V(a_2,b_2). 
\end{align*}

We now have the following result, which encapsulates the relevant
strong Lagrangian theory.

\begin{theorem}
  \label{thm:exist}
  Let $s^*$ denote the solution to the
  problem~$\Ps$.  Then there exists a vector
  $\lambda^*=(\lambda^*_1,\dots,\lambda^*_T)$ such that
  \begin{compactenum}[(i)]
  \item \label{con1} for all vectors~$s$ such that $s_0=s^*_0$ and
    $x_t(s)\in\X_t$ for all $t$ ($s$ is not otherwise constrained),
    \begin{equation}
      \label{eq:21}
      \sum_{t=1}^T \left[C_t(x_t(s)) + A_t(s_t) - \lambda^*_t s_t \right]
      \ge
      \sum_{t=1}^T \left[C_t(x_t(s^*)) + A_t(s^*_t)-  \lambda^*_t s^*_t \right].
    \end{equation}
  \item \label{con2} the pair $(s^*,\lambda^*)$ satisfies the
    complementary slackness conditions, for $1\le t\le T$,
    \begin{equation}
      \label{eq:22}
      \begin{cases}
        \lambda^*_t = 0 & \quad\text{if $0 < s^*_t < E_t$,}\\
        \lambda^*_t \ge 0 & \quad\text{if $s^*_t = 0$,}\\
        \lambda^*_t \le 0 & \quad\text{if $s^*_t = E_t$.}
      \end{cases}
    \end{equation}
  \end{compactenum}

  Conversely, suppose that there exists a pair of vectors
  $(s^*,\,\lambda^*)$, with $s_0=s^*_0$, satisfying the
  conditions~\eqref{con1} and \eqref{con2} and such that $s^*$ is
  additionally feasible for the problem~$\Ps$.  Then $s^*$
  solves the problem~$\Ps$.
\end{theorem}

\begin{proof}
  Consider the general problem~$\mathbf{P}(a,\,b)$ defined above.
  Introduce slack (or surplus) variables $z=(z_1,\dots,z_t)$ and
  $w=(w_1,\dots,w_t)$ and rewrite $\mathbf{P}(a,\,b)$ as:
  \begin{compactitem}
  \item[$\mathbf{P}(a,\,b)$:] minimise
    \begin{math}
      \sum_{t=1}^T [C_t(x_t(s)) + A_t(s_t)]
    \end{math}
    over all $s=(s_0,\dots,s_T)$ with $s_0=s^*_0$, over all $z\ge0$,
    over all $w\ge0$, and subject to the further constraints
    \begin{align}
      s_t - z_t & = a_t, \qquad 1 \le t \le T, \label{eq:23}\\
      s_t + w_t & = b_t, \qquad 1 \le t \le T, \label{eq:24}
    \end{align}
    and also $x_t(s)\in\X_t$ for $1 \le t \le T$. 
  \end{compactitem}
  Since the function $V(a,\,b)$ is also convex in $a$ and $b$, it
  follows from the supporting hyperplane theorem (see \cite{Boyd} or
  \cite{Whi}), that there exist Lagrange multipliers
  $\alpha^*=(\alpha^*_1,\dots,\alpha^*_T)$ and
  $\beta^*=(\beta^*_1,\dots,\beta^*_T)$ such that
  \begin{equation}
    \label{eq:25}
    V(a,\,b) \ge V(a^*,\,b^*) + \sum_{t=1}^T\alpha^*_t(a_t-a^*_t) +
    \sum_{t=1}^T\beta^*_t(b_t-b^*_t)
    \qquad\text{for all $a$, $b$}
  \end{equation}
  Thus also, for all $s$ with $s_0=s^*_0$ and such that $x_t(s)\in\X_t$
  for $1 \le t \le T$, for all $z\ge0$, and for all $w\ge0$, by
  defining $a$ and $b$ via \eqref{eq:23} and \eqref{eq:24}, we have
  \begin{multline}
    \label{eq:26}
    \sum_{t=1}^T \left[C_t(x_t(s)) + A_t(s_t) - \alpha^*_t(s_t-z_t)
      -\beta^*_t(s_t+w_t)\right]\\
    \ge
    \sum_{t=1}^T \left[C_t(x_t(s^*)) + A_t(s^*_t) - \alpha^*_ta^*_t
      -\beta^*_tb^*_t\right].
  \end{multline}
  Since the components of $z$ and $w$ may take arbitrary positive
  values, we obtain at once the following complementary slackness
  conditions for the vectors of Lagrange multipliers $\alpha^*$ and
  $\beta^*$:
  \begin{align}
    \alpha^*_t \ge 0, \qquad &
    \text{$\alpha^*_t=0$  whenever $s^*_t>a^*_t$},
    \qquad 1\le t \le T,
    \label{eq:27}\\
    \beta^*_t \le 0, \qquad &
    \text{$\beta^*_t=0$  whenever $s^*_t<b^*_t$},
    \qquad 1\le t \le T.
    \label{eq:28}
  \end{align}
  Thus, from \eqref{eq:26}--\eqref{eq:28}, by taking $z_t=w_t=0$ for
  all $t$ on the left side of~\eqref{eq:26}, it follows that, for all
  $s$ with $s_0=s^*_0$ and $x_t(s)\in\X_t$ for $1 \le t \le T$,
  \begin{equation}
    \label{eq:29}
    \sum_{t=1}^T
    \left[C_t(x_t(s)) + A_t(s_t) - (\alpha^*_t+\beta^*_t)s_t \right]
    \ge
    \sum_{t=1}^T
    \left[C_t(x_t(s^*)) +A_t(s^*_t) - (\alpha^*_t+\beta^*_t)s^*_t \right].
  \end{equation}
  The condition~\eqref{con1} of the theorem now follows on defining
  \begin{equation}
    \label{eq:30}
    \lambda^*_t=\alpha^*_t+\beta^*_t,
    \qquad 1 \le t \le T.
  \end{equation}
  while the condition~\eqref{con2} follows from \eqref{eq:30} on using
  also the complementary slackness conditions \eqref{eq:27} and
  \eqref{eq:28}.

  To prove the converse result, suppose that a pair
  $(s^*,\,\lambda^*)$ (with $s_0=s^*_0$) satisfies the
  conditions~\eqref{con1} and \eqref{con2} and that $s^*$ is feasible
  for the problem~$\Ps$.  From the condition~\eqref{con2}, we may
  define (unique) vectors $\alpha^*=(\alpha^*_1,\dots,\alpha^*_T)$ and
  $\beta^*=(\beta^*_1,\dots,\beta^*_T)$ such that the
  conditions~\eqref{eq:27}, \eqref{eq:28} and \eqref{eq:30} hold.  The
  condition~\eqref{con1} of the theorem now translates to the
  requirement that, for all vectors~$s$ such that $s_0=s^*_0$ and
  $x_t(s)\in\X_t$ for all $t$, the relation~\eqref{eq:29} holds.
  Finally, it follows from this and from the conditions~\eqref{eq:27}
  and \eqref{eq:28} that, for any vector $s$ which is feasible for the
  problem~$\Ps$---and so in particular satisfies $0\le s_t\le E_t$ for
  all $t$,
  \begin{equation}
    \label{eq:31}
    \sum_{t=1}^T \left[C_t(x_t(s)) + A_t(s_t) \right]
    \ge
    \sum_{t=1}^T \left[C_t(x_t(s^*)) +A_t(s^*_t) \right],
  \end{equation}
  so that $s^*$ solves the problem~$\Ps$ as required.
\end{proof}

\begin{remark}
  Note that the second part of Theorem~\ref{thm:exist}, i.e.\ the
  converse result, does not require the convexity assumptions on the
  functions~$C_t$ and $A_t$.
\end{remark}

The above Lagrangian theory---which we require for the determination
of the optimal control as described in
Section~\ref{sec:algorithm}---further enables a determination of the
sensitivity of the value of the store with respect to variation of its
capacity constraints.  For the given problem~$\Ps$, the cost of
optimally operating the store (the negative of its value) is given by
$V(a^*,\,b^*)$, where we recall that $a^*$ and $b^*$ are as given by
\eqref{eq:20}.  For any~$t$, the derivative of this optimised cost
with respect to $E_t$, assuming this derivative to exist, is given by
the Lagrange multiplier~$\beta^*_t$ defined in the above proof (the
differentiability assumption ensuring that $\beta_t$ is here uniquely
defined).  Note further that when $s^*_t<E_t$ then
(from~\eqref{eq:28}) the Lagrange multiplier~$\beta^*_t$ is equal to
zero, and when $s^*_t=E_t$ then (from \eqref{eq:27} and \eqref{eq:30})
we have $\beta^*_t=\lambda^*_t$.

A further determination of the sensitivity of the value of the store
with respect to variation of its rate constraint may be developed
along the lines of Theorem~5 of Cruise et al~\cite{CFGZ}, but we do
not pursue this here.

\section{Determination of $(s^*,\lambda^*)$}
\label{sec:algorithm}

The structure of the objective function causes some difficulties for
the solution of the problem~$\Ps$.  As previously observed, a dynamic
programming approach might seem natural but, even for this
deterministic problem, typically remains too computationally
complex---on account of both the likely time-heterogeneity of the
functions $A_t$ and $C_t$, and of the need, even for small $t$, to
consider the problem over the entire time interval~$[0,\,T]$.

We continue to assume convexity of the functions $C_t$ and $A_t$.
Under the further assumption of differentiability of the functions
$A_t$, we give an efficient algorithm for the construction of a pair
$(s^*,\lambda^*)$ satisfying the conditions of
Theorem~\ref{thm:exist}---so that, in particular, $s^*$ solves the
problem~$\Ps$.  This algorithm is further sequential and local in
time, in the sense that the determination of the solution to any given
time~$t'\le T$ typically requires only the consideration of the
problem, i.e.\ a knowledge of the functions $C_t$ and $A_t$, for those
times~$t$ extending to some time horizon which is typically only a
short distance beyond~$t'$.
We have already shown in Section~\ref{sec:simpl-optim-probl} that the
ability to dynamically solve the deterministic problem~$\Ps$, or
updates of this problem, at the times of successive shocks enables an
efficient (stochastically) optimal control of the store.

We give conditions necessarily satisfied by the pair
$(s^*,\lambda^*)$.  Under the further assumption of strict convexity
of the functions~$C_t$, we show how these conditions may be used to
determine $(s^*,\lambda^*)$ uniquely.  We then indicate how the strict/
convexity assumption may be relaxed.

\begin{proposition}
  \label{prop:existnu}
  Suppose that the functions~$A_t$ are differentiable, and that the
  pair~$(s^*,\lambda^*)$ is such that $s^*$ is feasible for the
  problem~$\Ps$, while $(s^*,\lambda^*)$ satisfies the
  condition~(\ref{con2}) of Theorem~\ref{thm:exist}.  For each~$t$
  define
  \begin{equation}
    \label{eq:32}
    \nu^*_t = \sum_{u=t}^T[\lambda^*_u - A'_u(s^*_u)].
  \end{equation}
  Then the condition that $(s^*,\lambda^*)$ satisfies the
  condition~(\ref{con1}) of Theorem~\ref{thm:exist} is equivalent to the
  condition that
  \begin{equation}
    \label{eq:33}
    \text{$x_t(s^*)$ minimises $C_t(x)-\nu^*_tx$ in $x\in\X_t$},
    \qquad 1 \le t \le T.
  \end{equation}
\end{proposition}

\begin{proof}
  Assume that the pair $(s^*,\lambda^*)$ is as given.  Suppose first
  that additionally $(s^*,\lambda^*)$ satisfies the
  condition~(\ref{con1}) of Theorem~\ref{thm:exist}.  The
  condition~\eqref{eq:33} is then straightforward when the
  functions~$C_t$ are additionally differentiable: for each~$t$ the
  partial derivative of the left side of \eqref{eq:21} with respect to
  $x_t(s)$ (with $x_u(s)$ being kept constant for $u\neq t$) is
  necessarily zero at $s=s^*$, so that~\eqref{eq:33} follows from the
  assumed convexity of the functions~$C_t$.  For the general case,
  note that it follows from the condition~\eqref{con1} of
  Theorem~\ref{thm:exist} (by considering $s$ such that $s_0=s^*_0$ ,
  $x_t(s)=x_t(s^*)+h$, $x_u(s)=x_u(s^*)$ for $u\ne t$), that, for
  all~$t$ and for all real~$h$,
  \begin{equation}
    \label{eq:34}
    C_t(x_t(s^*)+h) + \sum_{u=t}^T[A_u(s^*_u+h) - \lambda^*_uh] 
  \end{equation}
  is minimised at $h=0$, and so, for all (small)~$h$,
  \begin{equation}
    \label{eq:35}
    C_t(x_t(s^*)+h) - \nu_th \ge C_t(x_t(s^*)) + o(h),
    \quad\text{as $h\to0$}.
  \end{equation}
  Thus~\eqref{eq:33} again follows from the assumed convexity of the
  functions~$C_t$.

  To prove the converse result, suppose now that $(s^*,\lambda^*)$
  satisfies the condition~\eqref{eq:33}.  This condition, together
  with the convexity and differentiability of the functions~$A_t$,
  then implies that, for all~$t$, the expression \eqref{eq:34} is
  minimised at $h=0$.
  It is now straightforward that the hyperplane in $\R^T$ whose vector
  of slopes is $\lambda^*$ supports the function
  $\sum_{t=1}^T\left[C_t(x_t(s))+A_t(s_t)\right]$ at the point
  $(s^*,\,\sum_{t=1}^T[C_t(x_t(s^*))+A_t(s^*_t)])$, so that finally
  the condition~\eqref{con1} of Theorem~\ref{thm:exist} holds as
  required.
\end{proof}

It now follows from Theorem~\ref{thm:exist} and
Proposition~\ref{prop:existnu} that if the pair~$(s^*,\lambda^*)$ is
such that $s^*$ is feasible for the problem~$\Ps$, and that
$(s^*,\lambda^*)$ satisfies the both condition~\eqref{eq:33} and the
condition~(\ref{con2}) of Theorem~\ref{thm:exist}, then $s^*$ further
solves the problem~$\Ps$.

We now show how to construct such a pair $(s^*,\lambda^*)$.  We
assume, for the moment, strict convexity of the functions~$C_t$; we
subsequently indicate how to relax this assumption.  It follows from
the assumed strict convexity that, for each $t$ and for each $\nu_t$,
there is a unique $x\in\X_t$, which we denote by $x^*_t(\nu_t)$, which
minimises $C_t(x)-\nu_tx$ in $\X_t$.  Further $x^*_t(\nu_t)$ is
continuous and increasing in $\nu_t$---strictly so for $\nu_t$ such
that $x^*_t(\nu_t)$ lies in the interior of $\X_t$.  In particular,
from~\eqref{eq:13}, the condition~\eqref{eq:33} may now be rewritten
as
\begin{equation}
  \label{eq:36}
  s^*_t = s^*_{t-1} + x^*_t(\nu^*_t),
  \qquad 1 \le t \le T.
\end{equation}
It further follows from \eqref{eq:32} that
\begin{equation}
  \label{eq:37}
  \nu^*_{t+1} = \nu^*_t + A'_t(s^*_t) - \lambda^*_t.
  \qquad 1 \le t \le T-1.
\end{equation}

Thus, were the vector~$\lambda^*$ known, together with the value of
the constant~$\nu^*_1$, the pair $(s^*,\,\nu^*)$ could be constructed
sequentially via \eqref{eq:36} and \eqref{eq:37}.  We observe that,
while $\lambda^*$ is not known, it does satisfy the
conditions~\eqref{eq:22} and in particular the requirement that
$\lambda^*_t=0$ for all $t$ such that $0<s^*_t<E_t$.  We now follow a
procedure which is a generalisation of one described by Cruise et
al~\cite{CFGZ}, and which involves an essentially one-dimensional
search so as to identify the constant~$\nu^*_1$.  This search, which
may be thought of as being carried out at time zero and which is not
computationally intensive (see the further remarks at the end of this
section), then needs to be repeated at each of a number of subsequent
times as described below.  We show how to define inductively a
sequence of times $0=T_0<T_1<\dots<T_k=T$ such that $s^*(T_i)=0$ or
$s^*(T_i)=E_{T_i}$ for $1\le i\le k$ and such that $\lambda^*_t=0$ for
all values of $t$ not in the above sequence.

The time $T_1$ is chosen as follows.  Consider \emph{trial} values
$\nu_1$ of $\nu^*_1$.
For each such $\nu_1$, define a pair of vectors
$\nu=(\nu_1,\dots,\nu_T)$ and $s=(s_1,\dots,s_T)$ by
\begin{align}
  s_t & = s_{t-1} + x^*_t(\nu_t),
  \qquad 1 \le t \le T, \label{eq:38}\\
  \nu_{t+1} & = \nu_t + A'_t(s_t),
  \qquad 1 \le t \le T-1. \label{eq:39}
\end{align}

Define $M$ and $M'$ to be the sets of values of $\nu_1$ for which the
vector $s$ defined via \eqref{eq:38} and \eqref{eq:39} violates one of
the capacity constraints~\eqref{eq:15} and first does so respectively
below or above---in either case at a time which we denote by
$\overline T_1(\nu_1)$.  Since, for each $t$, $x^*_t(\nu_t)$ is
increasing in $\nu_t$ and $A'_t(s_t)$ is increasing in $s_t$ (by the
convexity of $A_t$), it follows that if $\nu_1\in M$ then
$\nu'_1\in M$ for all $\nu_1'<\nu_1$ and that if $\nu_1\in M'$ then
$\nu_1'\in M'$ for all $\nu_1'>\nu_1$; further the sets $M$ and $M'$
are disjoint, and (since the solution set for the problem~$\Ps$ is
nonempty) neither the set $M$ nor the set $M'$ can be the entire real
line.  Let $\bar\nu_1=\sup M$.  (In the extreme case where $M$ is
empty we may set $\bar\nu_1=-\infty$).  We now consider the behaviour
of the corresponding vector~$s$ defined via \eqref{eq:38} and
\eqref{eq:39} where we take $\nu_1=\bar\nu_1$; for this vector~$s$
there are three possibilities:
\begin{compactenum}[(a)]
\item the quantity $\bar\nu_1$ belongs neither to the set $M$ nor to
  the set $M'$, i.e.\ the vector $s$ generated as above is feasible
  for the problem~$\Ps$; in this case we take $T_1=T$ and $s^*=s$ with
  $\nu^*_1=\bar\nu_1$ and $\lambda^*_t=0$ for $1\le t\le T-1$ (so that
  the remaining values of $\nu^*$ are given by \eqref{eq:37});
\item the quantity $\bar\nu_1$ belongs to the set $M$; in this case
  there exists at least one $t<\overline T_1(\bar\nu_1)$ such that
  $s_t=E_t$ (were this not so then, by the continuity of each
  $x^*_t(\nu_t)$ in $\nu_t$, the value of $\nu_1$ could be increased
  above $\bar\nu_1$ while remaining within the set $M$); define $T_1$
  to be any such $t$, say the largest, and take $s^*_t=s_t$ for
  $1\le t\le T_1$ with $\nu^*_1=\bar\nu_1$ and $\lambda^*_t=0$ for
  $1\le t\le T_1-1$;
\item the quantity $\bar\nu_1$ belongs to the set $M'$; in this case,
  similarly to the case (b), there exists at least one
  $t<\overline T_1(\bar\nu_1)$ such that $s_t=0$; define $T_1$ to be
  any such $t$, again say the largest, and again take $s^*_t=s_t$ for
  $1\le t\le T_1$ with $\nu^*_1=\bar\nu_1$ and $\lambda^*_t=0$ for
  $1\le t\le T_1-1$.
\end{compactenum}

In each of the cases~(b) and (c) above, we now repeat the above
procedure, starting at the time~$T_1$ instead of the time~$0$, and
considering trial values of $\nu^*_{T_1+1}$, thereby identifying
$\nu^*_{T_1+1}$, the time~$T_2$ and the values of $s^*_t$ for
$T_1+1\le t\le T_2$, and taking $\lambda^*_t=0$ for
$T_1+1\le t\le T_2-1$.  The quantity~$\lambda^*_{T_1}$ is now defined
via \eqref{eq:37}.  Further consideration of the sets $M$ and $M'$
defined above in relation to the identification of $\nu^*_1=\bar\nu_1$
shows easily that in the case $\bar\nu_1\in M$---so that
$s^*_{T_1}=E_{T_1}$---the quantity $\nu^*_{T_1+1}=\bar\nu_{T_1+1}$ is
necessarily such that $\lambda^*_{T_1}\ge0$ (since in this case, by
the above construction, the quantity $\nu^*_{T_1+1}$ has a value which
is necessarily at least as great as would have been the case had
$\lambda^*_{T_1}$ been equal to $0$), whereas in the case
$\bar\nu_1\in M'$---so that $s^*_{T_1}=0$---the quantity
$\nu^*_{T_1+1}=\bar\nu_{T_1+1}$ is necessarily such that
$\lambda^*_{T_1}\le0$.

For $T_2\neq T$ we continue in this manner until the entire sequence
$0=T_0<T_1<\dots<T_k=T$ is identified.  We thus obtain vectors~$s^*$,
$\lambda^*$ and $\nu^*$ such that $s^*$ is feasible for the
problem~$\Ps$, while $(s^*,\lambda^*)$ satisfies the
condition~\eqref{eq:33} and the condition~(\ref{con2}) of
Theorem~\ref{thm:exist} and so solves the problem~$\Ps$ as required.

In the case where, for at least some $t$, the cost function $C_t$ is
convex, but not necessarily strictly so, some extra care is required.
Here, for such $t$, the function $\nu\rightarrow x^*_t(\nu)$ is not in
general uniquely defined; further, for any given choice, this function
is not in general continuous.  However, the above construction of
$(s^*,\lambda^*)$ continues to hold provided that, where necessary, we
choose the right value of $x^*_t(\nu)$.  The latter may always be
identified by considering, for example, a sequence of strictly convex
functions $C^{(\epsilon)}_t$ converging to $C_t$ and identifying
$x^*_t(\nu)$ as the limit of its corresponding values within this
sequence.

Note that the above construction proceeds locally in time, in the
sense that, at each successive time~$T_i$, the determination of the
subsequent time~$T_{i+1}$ and of the values of $s^*_t$ and $\nu^*_t$
for $T_i+1\le t\le T_{i+1}$ only requires consideration of the
functions~$C_t$ and $A_t$ up to some time $\overline T_{i+1}$
(necessarily beyond $T_{i+1}$) the identification of which does not
depend on the functions $C_t$ and $A_t$ at any subsequent times.  More
precisely we have $\overline T_1=\overline T_1(\bar\nu_1)$, where
$\overline T_1(\bar\nu_1)$ is as identified above, and the remaining
$\overline T_i$, $2\le i\le k$, are similarly identified.  In
particular we have that, for each time~$t$ and given $s^*_{t-1}$, the
optimal choice of store level~$s^*_t$ depends only on the
functions~$C_{t'}$ and $A_{t'}$ for $t\le t'\le\overline T(t)$ where
we define $\overline T(t)=\overline T_{i+1}$ for $i$ such that
$T_i+1\le t\le T_{i+1}$.  The function $\overline T(t)$ is piecewise
constant in~$t$, and so the \emph{time horizon} or \emph{look-ahead
  time} $\overline T(t)-t$ required for the optimal decision at each
time~$t$ has the ``sawtooth'' shape which we illustrate in our
examples of Section~\ref{sec:example}.

Note further that a lengthening of the total time $T$ over which the
optimization is to be performed does not in general change the values
of the times $T_i$, but rather simply creates more of them.  In
particular the solution to the problem~$\Ps$ involves computation
which grows essentially linearly in $T$, and the algorithm is suitable
for the management of a store with an infinite time horizon.

The typical length of the intervals between the successive times $T_i$
depends on the shape of the functions $C_t$ and $A_t$ and in
particular on the rate at which they fluctuate in time.  Thus, for
example, the long-run management of a store for which the
functions~$C_t$ show strong daily fluctuations typically involves
decision making on a running time horizon of the order of a day or so.

Finally note that, as already indicated, in the implementation of the
above construction, some form of one-dimensional search is usually
required to determine each of the successive $\bar\nu_{T_i+1}$: each
trial value of this quantity provides either an upper or lower bound
to the true value, so that, for example, a simple binary search is
sufficient.   Given also the ``locality'' property referred to above,
the numerical effort involved in the implementation of the above
algorithm is usually very slight.

\section{Examples}
\label{sec:example}

We give some examples, in which we solve (exactly) the optimal control
problem $\Ps$ formally defined in Section~\ref{sec:simpl-optim-probl}.
We investigate how the optimal solution depends on the cost
functions~$C_t$ defined there which reflecting buying and selling
costs and hence the opportunity to make money from price arbitrage,
and on the functions~$A_t$ which reflect the costs of providing
buffering services.

The cost functions~$C_t$ are derived from half-hourly electricity
prices in the Great Britain spot market over the entire year 2011,
adjusted for a modest degree of market impact, as described in detail
below.  Thus we work in half-hour time units, with the time
horizon~$T$ corresponding to the number of half-hour periods in the
entire year.  These spot market prices show a strong daily cyclical
behaviour (corresponding to daily demand variation), being low at
night and high during the day.  This price variation can be seen in
Figure~\ref{fig:price} which shows half-hourly GB spot prices (in
pounds per megawatt-hour) throughout the month of March 2011.  There
is a similar patter of variation throughout the rest of the year.

\begin{figure}
  \centering
  \includegraphics[scale=0.7]{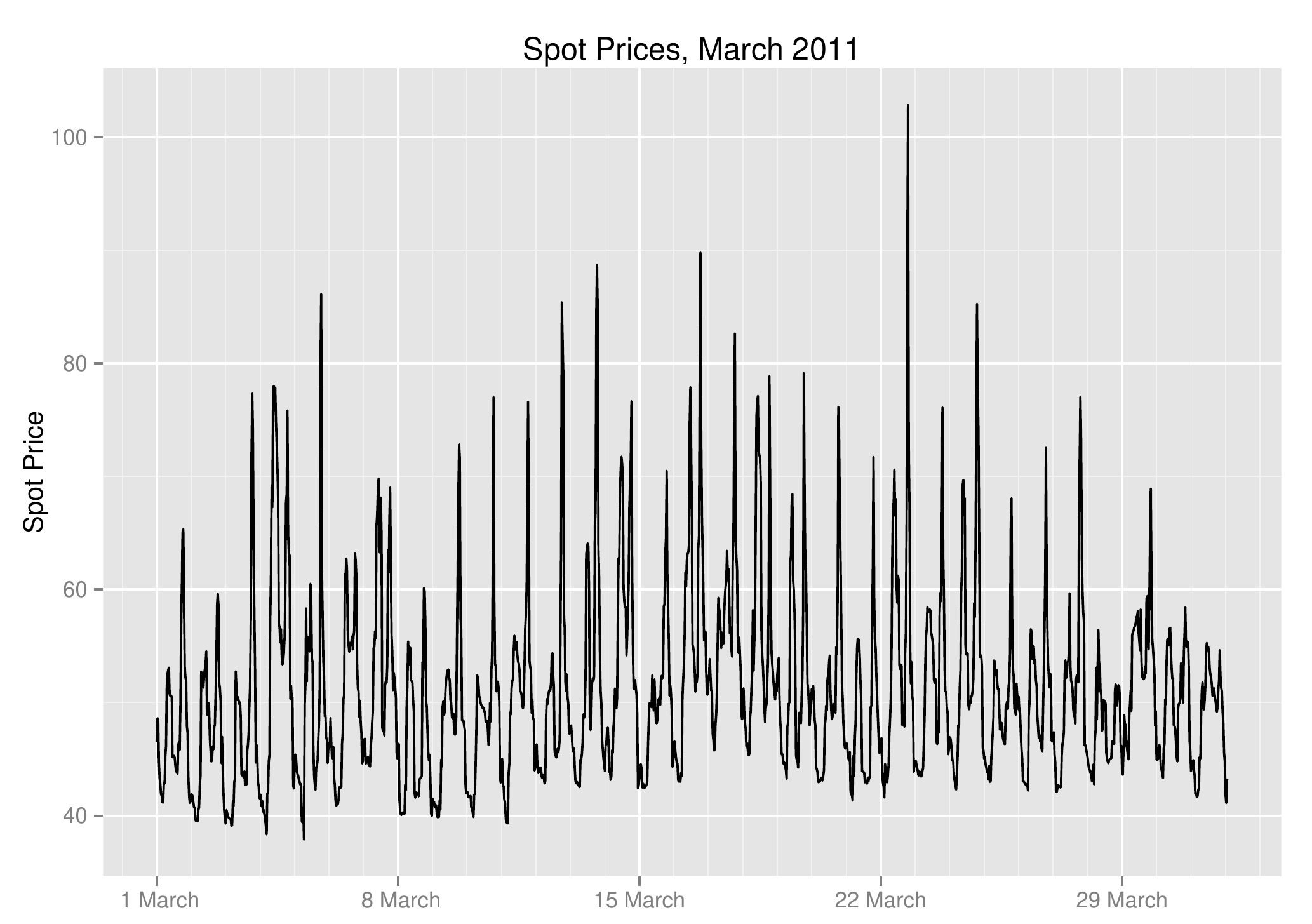}
  \caption{GB half-hourly spots prices (\pounds/MWh) for March 2011.}
  \label{fig:price}
\end{figure}

Without loss of generality, we choose energy units such that the rate
(power) constraints are given by $P_{It}=P_{Ot}=1$ unit of energy per
half-hour period.  For illustration, we take the capacity of the store
to be given by $E=10$ units of energy; thus the store can completely
fill or empty over a 5-hour period, which is the case, for example,
for the large Dinorwig pumped storage facility in Snowdonia \cite{Din}.

We choose cost functions $C_t$ of the form
\begin{equation}
  \label{eq:40}
  C_t(x) =
  \begin{cases}
    c_t x (1+\delta x), & \quad\text{if $x\ge0$}\\
    \eta c_t x (1+\delta x), & \quad\text{if $x<0$},
  \end{cases}
\end{equation}
where the $c_t$ are proportional to the half-hourly electricity spot
prices referred to above, where $\eta$ is an adjustment to selling
prices representing in particular round-trip efficiency as described
in Section~\ref{sec:problem}, and where the factor $\delta>0$ is
chosen so as to represent a degree of market impact (higher unit
prices as the store buys more and lower unit prices as the store sells
more).  For our numerical examples we take $\eta=0.85$ which is a
typical round-trip efficiency for a pumped-storage facility such as
Dinorwig.  We choose $\delta=0.05$; since the rate constraints for the
store are $P_{It}=P_{Ot}=1$ this corresponds to a maximum market
impact of~$5$\%.  While this is modest, our results are qualitatively
little affected as $\delta$ is varied over a wide range of values less
than one, covering therefore the range of possible market impact
likely to be seen for storage in practice.

Finally we need to choose the functions~$A_t$ reflecting the costs of
providing buffering services.  Our aim here is to give an
understanding of how the optimal control of the store varies according
to the relative economic importance of cost arbitrage and buffering,
i.e.\ according to the relative size of the functions~$C_t$ and $A_t$.
We choose functions~$A_t$ which are constant over time~$t$ and of the
form $A_t(s)=ae^{-\kappa s}$ and $A_t(s)=b/s$ for a small selection of
the parameters $a$, $\kappa$ and $b$.  The extent to which a store
might provide buffering services in applications is extremely varied,
and so the likely balance between arbitrage and buffering cannot be
specified in advance.  Rather we choose just sufficient values of the
above parameters to show the effect of varying this balance.  For a
possible justification of the chosen forms of the functions~$A_t$
(including why it should not necessarily be truncated to $0$ for
values of $s$ greater than the rate constraint of $1$), see
Section~\ref{sec:determ-funct-a_t-2}; in particular the form
$A_t(s)=ae^{-\kappa s}$ is plausible in the case of light-tailed
shocks, while the form $A_t(s)=b/s$ shows the effect of a slow rate of
decay in~$s$.

In each of our examples, we determine the optimal control of the store
over the entire year, with both the initial level~$S^*_0$ and the
final level~$S^*_T$ given by $S^*_0=S^*_T=0$.  In each of the
corresponding figures, the upper panel shows the optimally controlled
level of the store throughout the month of March.  The lower panel
shows, for each time~$t$ in the same month, the time horizon (or
look-ahead time) $\overline T(t)-t$, defined in
Section~\ref{sec:algorithm}, i.e.\ the length of time beyond the
time~$t$ for which knowledge of the cost functions is required in
order to make the optimal decision at time~$t$.

Figure~\ref{fig:exp} shows the optimal control of the store when the
functions~$A_t$ are given by $A_t(s)=ae^{-\kappa s}$.  The uppermost
panels correspond to $a=0$, so that the store incurs no penalty for
failing to provide buffering services and optimises its control solely
on the basis of arbitrage between energy prices at different times.
The daily cycle of prices is sufficiently pronounced that here the
store fills and empties---or nearly so---on a daily basis,
notwithstanding the facts that the round-trip efficiency of $0.85$ is
considerably less than $1$ and that the minimum time for the store to
fill or empty is $5$ hours.  It will be seen also that the time
horizon, or look-ahead time, required for the determination of optimal
decisions is in general of the order of one or two days.  

\begin{figure}
  \centering
  \includegraphics[scale=0.8]{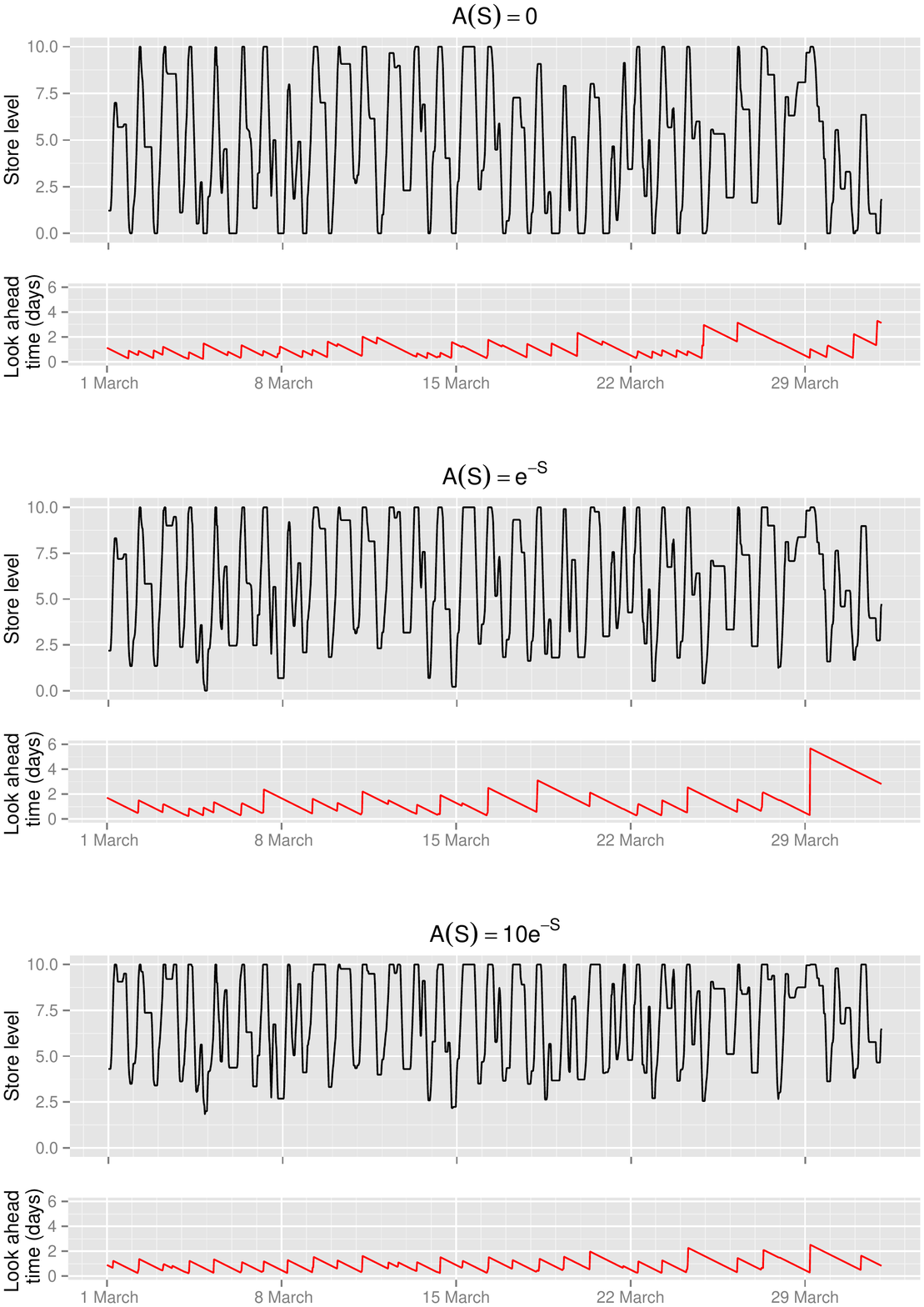}
  \caption{Store level and time horizon throughout March 2011 for the
    example with $A_t(s)=ae^{-\kappa s}$.  The top panels correspond
    to $a=0$, the central panels to $a=1$, $\kappa=1$, and the bottom
    panels to $a=10$, $\kappa=1$.}
  \label{fig:exp}
\end{figure}

The central panels of Figure~\ref{fig:exp} correspond to $\kappa=1$
and $a=1$.  The choice of $a$ in particular is such that the store is
just sufficiently incentivised by the need to reduce buffering costs
that it rarely empties completely (though it does so very
occasionally).  Otherwise the behaviour of the store is very similar
to that in the case $a=0$.  Note also that in this case the time
horizons or look-ahead times are in general somewhat longer; an
intuitive explanation (backed by a careful examination of the figure)
is that, starting from a time when the store is full, the
determination of by how much the store should avoid emptying
completely requires taking account of the cost functions for a longer
period of future time than is the case where the store does empty
completely.

Finally the bottom two panels of Figure~\ref{fig:exp} correspond to
$\kappa=1$ and $a=10$.  Here the costs of failing to provide buffering
services are much higher, and so the optimised level of the store
rarely falls below 25\% of its capacity.  Curiously the look-ahead
times are in general less than in the case $a=1$---presumably since
the store level is more often reaching the capacity constraint.

Variation of the exponential parameter~$\kappa$ does not result in
dramatically different behaviour, so we do not pursue this here.

Figure~\ref{fig:inv} shows the optimal control of the store when the
functions~$A_t$ are given, for each~$t$, by $A_t(s)=b/s$.  The upper
panels correspond to $b=0$, so that we again have $A_t(s)=0$ for all
$s$ and the control is as observed previously.  The lower panels
correspond to the case $b=1$, and, as might be expected, the behaviour
here is somewhat intermediate between that for the two nonzero
exponentially decaying exponential functions.

\begin{figure}[!h]
  \centering
  \includegraphics[scale=0.8]{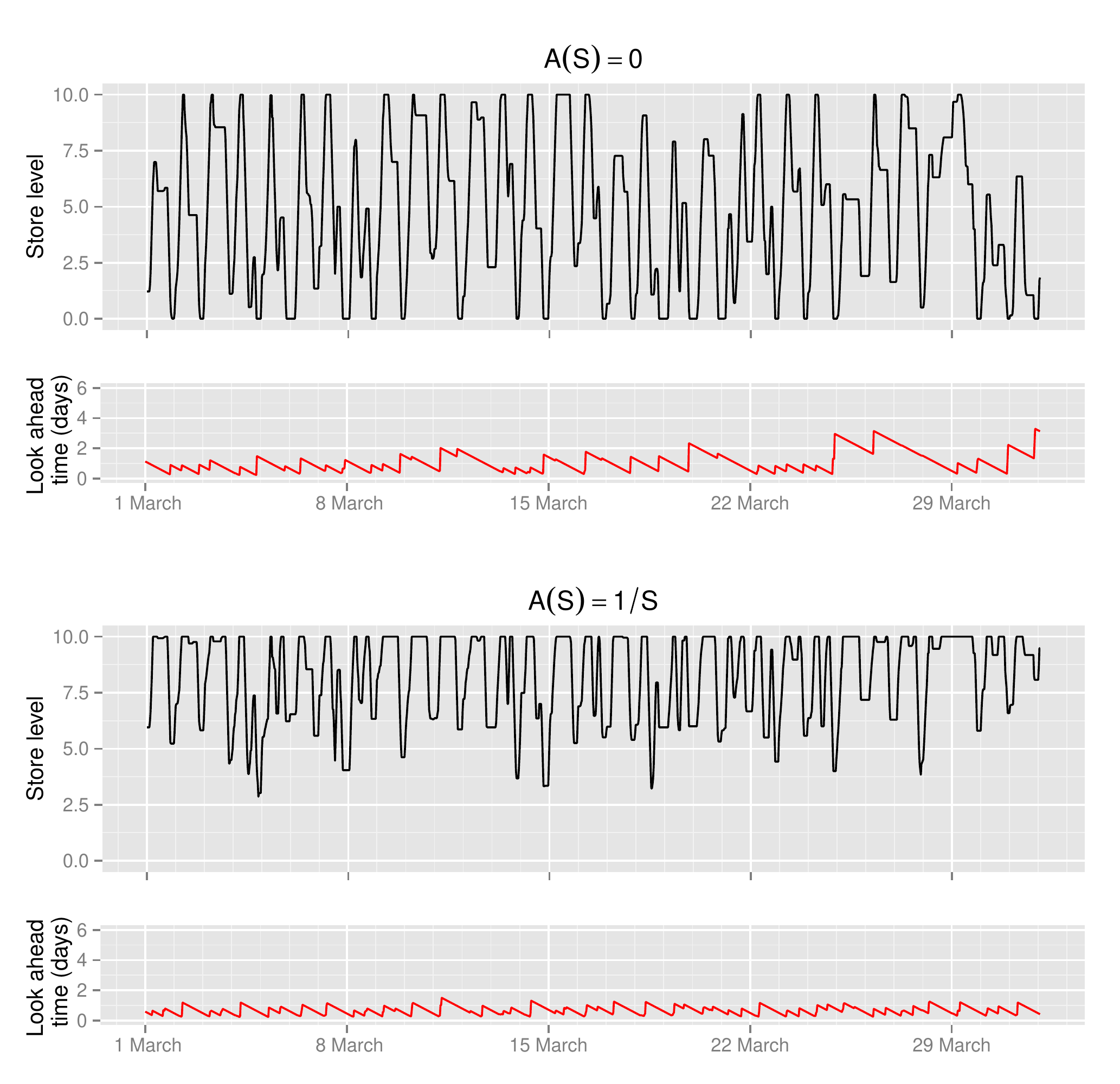}
  \caption{Store level and time horizon throughout March 2011 for the
    example with $A_t(s)=b/s$.  The upper panels correspond to
    $A_t(s)= 0$ and the lower panels to $A(t)= 1/s$.}
  \label{fig:inv}
\end{figure}



\section*{Acknowledgements}
\label{sec:acknowledgements}
The authors acknowledge the support of the Engineering and Physical
Sciences Research Council for the programme (EPSRC grant EP/I017054/1)
under which the present research is carried out.  They also
acknowledge the benefit of very helpful discussions with numerous
colleagues, notably: Janusz Bialek, Chris Dent, Lisa Flatley, Richard
Gibbens, Frank Kelly and Phil Taylor.

\newpage  
\bibliographystyle{ieeetr}

\bibliography{refs}

\end{document}